\documentclass[12pt]{amsart}
\usepackage{amsmath,amsthm,amsfonts,amssymb,color}
\begin{document} 
\newcommand{\tensor}{\otimes}
\renewcommand{\H}{{\mathbb H}}
\newcommand{\Hc}{\mathop {{\mathbb H}_{\C}}}
\renewcommand{\S}{{\mathbb S}}
\newcommand{\A}{{\mathbb A}}
\newcommand{\B}{{\mathbb B}}
\newcommand{\C}{{\mathbb C}}
\newcommand{\N}{{\mathbb N}}
\newcommand{\Q}{{\mathbb Q}}
\newcommand{\QQ}[1]{\left<#1,#1\right>}
\newcommand{\Z}{{\mathbb Z}}
\renewcommand{\P}{{\mathbb P}}
\newcommand{\R}{{\mathbb R}}
\newcommand{\rc}{\subset}
\newcommand{\rank}{\mathop{rank}}
\newcommand{\trace}{\mathop{tr}}
\newcommand{\dimc}{\mathop{dim}_{\C}}
\newcommand{\Lie}{\mathop{Lie}}
\newcommand{\Spec}{\mathop{Spec}}
\newcommand{\Auto}{\mathop{{\rm Aut}_{\mathcal O}}}
\newcommand{\alg}[1]{{\mathbf #1}}
\newcommand{\XXX}{XXX}
\let\Realpart\Re
\renewcommand{\Re}{\mathop{\Realpart e}}
\let\Impart\Im
\renewcommand{\Im}{\mathop{\Impart m}}
\newtheorem{lemma}{Lemma}[section]
\newtheorem*{definition}{Definition}
\newtheorem*{claim}{Claim}
\newtheorem{corollary}{Corollary}
\newtheorem*{Conjecture}{Conjecture}
\newtheorem*{SpecAss}{Special Assumptions}
\newtheorem{example}{Example}
\newtheorem*{remark}{Remark}
\newtheorem*{observation}{Observation}
\newtheorem*{fact}{Fact}
\newtheorem*{remarks}{Remarks}
\newtheorem{proposition}[lemma]{Proposition}
\newtheorem{theorem}[lemma]{Theorem}

\numberwithin{equation}{section}
\def\labelenumi{\rm(\roman{enumi})}
\title{%
On a quaternionic Picard theorem
}
\author {Cinzia Bisi \& J\"org Winkelmann}
\begin{abstract}
  The classical theorem of Picard states that a non-constant holomorphic
  function $f:\C\to\C$ can avoid at most one value.

  We investigate how many values a
  non-constant slice regular function of a quaternionic variable
  $f:\H\to\H$ may avoid.
\end{abstract}
\address{%
Cinzia Bisi \\
Department of Mathematics and Computer Sciences\\
Ferrara University\\
Via Machiavelli 35\\
44121 Ferrara \\
Italy
}
\email{bsicnz@unife.it \newline
 ORCID : 0000-0002-4973-1053
}
\address{%
J\"org Winkelmann \\
IB 3/111\\
Lehrstuhl Analysis II \\
Fakult\"at f\"ur Mathematik\\
Ruhr-Universit\"at Bochum\\
44780 Bochum \\
Germany\\
}
\email{joerg.winkelmann@rub.de\newline
  ORCID: 0000-0002-1781-5842
}
\thanks{
The two authors were partially supported by GNSAGA of INdAM.
C. Bisi was also partially supported by PRIN \textit{Variet\'a reali e complesse:
geometria, topologia e analisi armonica}. 
}
\maketitle

\section{Introduction}

A function $f:\C\to\C$ which is given by a globally convergent
power series $f(z)=\sum_{k=0}^{\infty} a_kz^k$ ($a_k\in\C$)
is called an {\em entire function}.
By the theorem of Picard, a non-constant entire function $f:\C\to\C$ can
avoid at most one value, \cite{P1}, \cite{P2}, \cite{P3}.

Our goal is a similar statement for {\em entire slice regular functions}, i.e.,
for functions $f:\H\to \H$ (where $\H$ denotes the skew field
of quaternions) which are given as a globally convergent power series
$f(q)=\sum_{k=0}^{\infty} q^k a_k$ ($a_k\in\H$).

For a function $f:\H\to\H$ being ``slice regular'' is equivalent to
the assumption that for every imaginary unit $I\in\S$ its restriction
to $\C_I=\{x+yI:x,y \in\R\}$ is holomorphic
with respect to the complex structures induced
by left multiplication by $I$, see \cite{GSS,GS}.

Here we show the following
\begin{enumerate}
\item
  For every $2$-dimensional real affine subspace $P$ of $\H\simeq\R^4$,
  there exists an entire slice regular function $f:\H\to\H$ such that
  $f(\H)=\H\setminus P$. In particular, for
  every triple $q_1,q_2,q_3\in\H$ there is an entire slice regular
  function avoiding these three values.
\item
  Let $q_1,\ldots,q_5\in\H$ be in general position (i.e.~these five
  quaternions are not contained in any $3$-dimensional real affine subspace
  of $\H$). Then every entire slice regular function avoiding all these five
  values must be constant.
  In particular, for every non-constant entire slice regular function
  the image is dense
  in $\H$.
\end{enumerate}

We do not know whether an entire slice regular function may
avoid a generic choice of four quaternionic numbers.

 A key tool 
 is the following fundamental correspondence
 (see Proposition~\ref{2.2}):

 {\em
   Let $f$ be a slice regular function and $F$ its stem function.
  Let $x,y\in\R$ and $c\in\H$. Then there exists an imaginary unit
  $I\in\S$ such that $f(x+yI)=c$ if and only
  $F(x+yi)-c\tensor 1$ is a zero divisor in the algebra $\H\tensor\C$.
}
 \\
Maybe this work can be of some inspiration in studying Hyperbolic Quaternionic Slice Regular Manifolds : indeed recently  many examples  of quaternionic slice regular manifolds have been introduced, see for example \cite{BG},\cite{AnB}.

 \section{Preparations}
 \subsection{Quaternions}
 The quaternionic numbers are a real $4$-dimensional skew field $\H$,
 which may be described as the non-commutative $\R$-algebra with $1$
 generated by $I,J,K$ with $I^2=J^2=K^2=-1$,
 $K=IJ=-JI$, $I=JK=-KJ$ and  $J=KI=-IK$.

 The set of all elements $q\in \H$ with $q^2=-1$ is called
 the set of {\em imaginary units} and denoted by $\S$.

 One may check easily that
 \[
 \S=\{c_2I+c_3J+c_4K: c_i\in\R, \sum_{i=2}^4 c_i^2=1\}.
 \]
 
\subsection{Slice regular functions and Stem functions}
We recall the theory of slice regular functions and their stem functions
(\cite{GS}, \cite{ghiloniperotti}).

An {\em entire slice regular function} $f:\H\to\H$ is a function which is
given by a globally convergent power series $f(q)=\sum_{k=0}^\infty q^ka_k$
(with $a_k\in\H$).

A {\em stem function} $F$ is a holomorphic map from $\C$ to
the $\C$-algebra $\Hc=\H\tensor_{\R}\C$ such that $\overline{F(z)}=F(\bar z)$.
The tensor product $\H\tensor_{\R}\C$
inherits a complex structure from its second
factor, $\C$, hence it makes sense to talk about holomorphicity and
complex conjugation.

In explicit terms, the stem function $F$ associated to a slice
regular function
$f(q)=\sum_{k=0}^\infty q^ka_k$ may be defined as
$F(z)=\sum_{k=0}^\infty a_k\tensor z^k$.

Equivalently, the correspondence may be described as follows:
\[
F(x+yi)=F_1(x+yi)\tensor 1 +F_2(x+yi)\tensor i
\]
with
\[
F_1(x+yi)=\frac{1}{2}\left( f(x+yI)+f(x-yI) \right)
\]
and
\[
F_2(x+yi)=-\frac{1}{2}I\left( f(x+yI)-f(x-yI) \right).
\]
For a slice regular function $f$ the terms on the right hand side
can be shown to be independent of the choice of the imaginary unit $I$.

Conversely, one has
\[
f(x+yH)=F_1(x+yi)+HF_2(x+yi)\quad\forall x,y\in\R,H\in\S.
\]

\subsection{A remarkable quadric in $\Hc$}

The euclidean scalar product on $\H\simeq\R^4$ induces a complex
symmetric bilinear form $\left<\ ,\ \right>$ on $\Hc$.
Explicitly: $\left<z,w\right>=\sum_{i=1}^4  z_iw_i$.

We observe that $\Hc$ naturally carries the structure of
an $\R$-algebra.

Both the field of complex numbers $\C$ and the quaternionic
skew field $\H$ embed into $\Hc$ via $z\mapsto 1\tensor z$
resp.~$q\mapsto q\tensor 1$. In this way we may regard
$\C$ and $\H$ as subrings of $\Hc$.

\begin{proposition}\label{prop-quadric}
  Let $v=1\tensor v_0 + I\tensor v_1 + J\tensor v_2
  + K\tensor v_3=v'\tensor 1+v''\tensor i$ (with $v_i \in\C$, $v',v''\in\H$)
  be an element of $\Hc$.

  Then the following are equivalent:
  \begin{enumerate}
  \item
    $v$ is a zero divisor, i.e., there exists an element $w\in\Hc$,
    $w\ne 0$ with $w\cdot v=0$,
  \item
    $\QQ{v}=0$, i.e., 
    $\sum_{i=0}^3 v_i^2=0$.
  \item
    There exists an imaginary unit $H\in\S$ such that $Hv'=v''$.
    (Geometrically: The vectors $v'$ and $v''$ are orthogonal.)
  \end{enumerate}
\end{proposition}

The above equivalence $(i)\iff(ii)$ is contained in \cite{S} where it is
attributed to Hamilton, while $(ii)\iff(iii)$ may be deduced from the
work of Mongodi (\cite{M}). In addition, these equivalences may be obtained
as a special case of a result of Ghiloni and Perotti
(\cite{ghiloniperotti}, Theorem~17 on page 1679).

For the convenience of the reader we nevertheless give a proof
here.

\begin{proof}
  $\text{(i)}\implies\text{(iii)}$:
  We assume that $v$ is a zero divisor (but $v\ne 0$).
  Since $\H$ has no zero divisors, it follows that $v',v''\ne 0$.
  Now $v'\in\H^*$ and $v$ being a zero divisor, imply that
  $v\cdot\left( (v')^{-1}\tensor 1\right)$ is again a zero divisor.
  Hence we may assume that $v'=1$. The same reasoning also shows
  that we can find an element $w=w'+w''\tensor i$ with
  $w'=1$ and $w\cdot v=0$.
  Thus we obtain
  \[
  0=w\cdot v=(1+w''\tensor i)\cdot(1+v''\tensor i)
  = (1 - w''v'')\tensor 1 + (v''+w'')\tensor i.
  \]
  Hence $v''=-w''$ and $(v'')^2=-v''w''=-1$, i.e., $v''\in\S$.
  In particular $v''=H\cdot 1=H\cdot v'$ for some $H\in\S$.

  $\text{(iii)}\implies\text{(i)}$:
  We have $v=(1\tensor 1+H\tensor i)\cdot v'$.
  Define $w=1\tensor 1 -H\tensor i$. Then $w\cdot v=0$, as easily seen
  by explicit calculation.

  $\text{(iii)}\iff\text{(ii)}$. Note that
  \[
  \QQ{v}=\QQ{v'+v''\tensor i}=\QQ{v'}-\QQ{v''}+2i\left<v',v''\right>.
  \]
  Hence $\QQ{v}=0$ iff $v'$ and $v''$ have the same norm and are
  orthogonal to each other. This in turn is equivalent to the
  existence of an imaginary unit $H\in\S$ with $v''=Hv'$.
\end{proof}

Thus the set of all zero divisors of $\Hc$ is a quadric subvariety
of $\Hc\simeq\C^4$.
This quadric has also been investigated by Mongodi (\cite{M}), who pointed out
the relevance for the zero locus, but not the relation with
zero divisors of the algebra $\Hc$.

\subsection{Zeroes}
Let $f$ be a slice function and let $F$ denote its stem function.
Write $F=F_1\tensor 1+F_2\tensor i$, with $F_h:\C\to\H$.
Since
\[
f(x+yI)=F_1(x+yi)+IF_2(x+yi)\ \ \forall x,y\in\R, I\in\S,
\]
implying
\[
f(x+yI)=0\ \iff\ F_1(x+yi)=-IF_2(x+yi),
\]
The following result is implied by Proposition~\ref{prop-quadric}, but
may also be deduced from \cite{M}, Proposition~4.1 in combination with
Corollary~3.4 of \cite{M}:

\begin{proposition}\label{2.2}
  Let $f:\H\to\H$ be a slice regular function and
  $F:\C\to\Hc=\H\tensor_{\R}\C$ its stem function.
  Let $x,y\in\R$. Then the following conditions are equivalent:
  \begin{enumerate}
  \item
    There exists an imaginary unit $H\in\S$ with $f(x+yH)=0$.
  \item
    $\QQ{F(x+yi)}=0$
  \item
    $F(x+yi)$ is a zero divisor in the algebra $\H\tensor_{\R}\C$.
  \end{enumerate}
\end{proposition}

This has the following consequence:
let $c\in\H$. Then a slice regular function $f$ avoids $c$
as value (i.e., $f(\H)\subset\H\setminus \{c\}$)
if and only if $z\mapsto F(z)-c\tensor 1$ has no zero which happens
if and only if
the entire function
\[
Q_c: z\mapsto\left< F(z)-c,F(z)-c\right>
=\left<F(z),F(z)\right>-2\left<F(z),c\right>+\left<c,c\right>
\]
has no zeroes.

\section{Avoiding $5$ generic values}

The purpose of this section is to show that a
non-constant entire slice regular
function can not avoid $5$ values if these are generic in the following
sense: there is no real $3$-dimensional affine subspace of $\mathbb{H}\simeq\R^4$
containing all of them.

We start with some preparations.

First we recall two results of Noguchi on holomorphic curves
in semi-abelian varieties. Here we do not need to deal with arbitrary
semi-abelian varieties, it suffices to know that $(\C^*)^g$ is a
semi-abelian variety.

\begin{proposition}[Logarithmic Bloch Ochiai Theorem]\label{LBO}
  Let $f:\C\to G=(\C^*)^g$ be a holomorphic map and let
  $X$ denote the Zariski closure of its image.

  Then $X$ is an orbit of  an algebraic subgroup $H$ of
  $G=(\C^*)^g$ (acting by left multiplication), i.e.,
  there is an element $\lambda=(\lambda_1,\ldots,\lambda_g)\in G=(\C^*)^g$
  such that
  \[
  X=\{\lambda \cdot h: h\in H\}
  \]
\end{proposition}

See Main Theorem $\text{(i)}$ in \cite{N}.

\begin{proposition}\label{NogD}
  Let
  \[
  f:\Delta^*=\{z\in\C:0<|z|<1\}\to G=(\C^*)^g
  \subset \bar G=(\P_1)^g
  \]
  be a holomorphic map
  and let $X$ denote the Zariski closure of its image.
  Define
  \[
  Stab(X)=\{g\in G: g\cdot x\in X\ \forall x \in X\}
  \]
  If $Stab(X)$ is discrete, then
  $f$ extends to a holomorphic map from $\Delta$ to $\bar G$.
\end{proposition}

\begin{proof}
This is a consequence of Theorem 4.5. of \cite{Nog2}, applied with
taking the Zariski closure of $f(\Delta^*)$ as $X$.
In the notation of \cite{Nog2} non-extendibility of $f$ implies
$f(\Delta^*)\subset W$. Since we take $X$ to be the Zariski closure of
the image of $f$, the inclusion $f(\Delta^*)\subset W$ implies $X=W$.
In view of Lemma~4.1 in \cite{Nog2} the condition $X=W$ implies
that $Stab(X)$ is not discrete.
\end{proof}
  
\begin{proposition}\label{NBO}
  Let $Z$ be an algebraic subvariety of $G=(\C^*)^5$.
  Assume that there exists a non-constant holomorphic
  map $g:\C\to Z$ with $g(z)=\overline{g(\bar z)}$
  for all $z\in\C$.

  Then there exist $\alpha_1,\ldots,\alpha_5\in\R^*$
  and $(m_1,\ldots,m_5)\in\Z^5\setminus\{(0,\ldots,0)\}$
  such that $\zeta(\C^*)\subset Z$
  for 
  \[
  \zeta(z)\stackrel{def}{=}
  \left( \alpha_1z^{m_1},\ldots,\alpha_5z^{m_5}\right).
  \]
\end{proposition}
\begin{proof}
  The Zariski closure of the image $g(\C)$ in $G$
  is an orbit of an algebraic subgroup  $H$ of $G$
  acting by multiplication (Proposition~\ref{LBO}).
  We choose a connected one-dimensional algebraic subgroup $T$ of $H$.
  Such a subgroup $T$ is isomorphic to $\C^*$ and 
  parametrized by a map
  $\zeta_0:\C^*\to G=(\C^*)^5$ given as
  \[
  \zeta_0(z)\stackrel{def}{=}
  \left( z^{m_1},\ldots, z^{m_5}\right).
  \]

  Define $\alpha=(\alpha_1,\ldots,\alpha_5)\stackrel{def}{=}g(0)$.
  The condition $g(z)=\overline{g(\bar z)}$ implies that
  $\alpha_i\in\R$ for all $i\in\{1,\ldots,5\}$.
  By our construction the $H$-orbit through $\alpha$ must be
  contained in $Z$. It follows that $\zeta(\C^*)\subset Z$
  for
  \[
  \zeta(z)=\zeta_0(z)\cdot\alpha=
  \left( \alpha_1z^{m_1},\ldots,\alpha_5z^{m_5}\right).
  \]
  \end{proof}
    
\begin{proposition}\label{3.1}
  Let $c_1,\ldots,c_4$ be a basis of the real vector space $\H$.
  Let   $M\in Mat(4\times 4,\R)$ be a positive definite
  symmetric real matrix.
  Let $Z$ denote the zero set of the function
  $\psi$ in $G=(\C^*)^5$ where
  \[
  \psi(v_1,\ldots, v_4; p)
  = p - w^tMw\quad\quad (v=(v_1, \cdots, v_4) \in \C^4, p\in\C)
  \]
  with
  \[ w= v - \begin{pmatrix}
  p + \QQ{c_1} \\ \vdots \\ p+\QQ{c_4}\\
  \end{pmatrix}.
\]

  Let $\alpha_i\in\R^*$ and $m_i\in\Z$ such that the image
  of the map $\zeta:\C^*\to G$ given as
  \[
  \zeta(z)\stackrel{def}{=}
  \left( \alpha_1z^{m_1},\ldots,\alpha_5z^{m_5}\right)
  \]
  is contained in $Z$ (i.e.~$\zeta(\C^*)\subset Z$).

  Then $m_i=0$ for all $i\in\{1,\ldots,5\}$, i.e.,
  $\zeta$ must be constant.
\end{proposition}

\begin{proof}
  We discuss the coefficients of the Laurent series
  $\sum_{k\in\Z}b_kz^k$ of the
  holomorphic function $z\mapsto (\psi\circ\zeta)(z)$
  defined on $\C^*$.
  Since $\psi\circ\zeta\equiv 0$ due to $\zeta(\C^*)\subset Z$,
  we know that $b_k=0$ for all $k\in\Z$.
  On the other hand, the Laurent coefficients $b_k$ depend on
  the matrix   $M$ and the coefficients $\alpha_i, m_i$.
  Using these facts we will see that we arrive at a contradiction
  if we assume that $\zeta$ is not constant.
  
  We start by observing that $\psi$ is a polynomial map of degree $2$
  whose purely quadratic term is given by
  \[
  \psi_2(v;p)= - (v-pd)^tM(v-pd)\quad\text{ with }d=(1,\ldots,1)^t
  \]
  
  We may replace $\zeta$ with its composition with the inverse
  element map $z\mapsto 1/z$ and thereby assume $m_5\ge 0$.
  By permuting variables we may also assume
  that
  \[
  m_1\le m_2\le m_3 \le m_4.
  \]
  Let us now assume that $\zeta$ is not constant, i.e.,
  let us assume that $(m_1,\ldots,m_5)\ne(0,\ldots,0)$.
  Our strategy is to show that the Laurent series of $\psi\circ
  \zeta$ can not vanish unless $(m_1,\ldots,m_5)=(0,\ldots,0)$.

  {\em Case 1.} We assume $m_1<0$.

  Fix $k$ such that
  $m_i=m_1$ for $1\le i\le k$ and $m_i>m_1$ for $k<i\le 4$.
  We consider the Laurent coefficient
  of degree $2m_1$. Note that $\zeta$ has no homogeneous component of
  degree less that $m_1$.
  Recall that $\psi$ is a quadratic polynomial.
  It follows that $\psi\circ\zeta$ has no homogeneous component of degree
  less that $2m_1$ and that the homogeneous component of degree $2m_1$
  equals $(\psi_2\circ\zeta)_{2m_1}$ where $\psi_2$ is the purely
  quadratic part of $\psi$ and $(\psi_2 \circ\zeta)_{2m_1}$
  is the homogeneous component of
  $\psi_2\circ\zeta$ of degree $2m_1$.
  Thus $(\psi_2\circ\zeta)_{2m_1}=b_{2m_1}z^{2m_1}$.

  By the definition of $\psi$ and $\zeta,$ it follows that
  $b_{2m_1}=-u^tMu$ with
  \[
  u=(\alpha_1,\ldots,\alpha_k,0,\ldots,0).
  \]
  But $M$ is positive definite and the $\alpha_i$ are all real
  and non-zero.
  Hence $u^tMu>0$, contradicting $\psi\circ\zeta\equiv 0$.

  {\em Case 2.} We assume $m_5>0$ and $m_1\ge 0$.
  
  Fix $k\in\{1,\ldots, 4\}$ such that $m_i=0$ iff $i\le k$.
  Here we investigate
  the constant term of the Laurent series of $\psi\circ\zeta$, i.e.,
  its degree-0-coefficient.
  
  This is $b_0=-u^tMu$ with
  \[
  u=\left(\alpha_1+\QQ{c_1},\ldots,\alpha_k+\QQ{c_k},\QQ{c_{k+1}},\ldots,
  \QQ{c_4}\right).
  \]
  We employ again the facts that $M$ is positive definite and $u$ is real.
  Hence $u^tMu=0$ requires that $u$ is the zero vector.
  Because $\QQ{c_i}>0$, it follows that $k=4$. Thus $m_i=0$ for
  all $i<5$. But now it follows that the degree $2m_5$-term
  is
  $-v^tMv$ with
  \[
  v=(\alpha_5,\ldots,\alpha_5)
  \]
  which yields a contradiction.

  {\em Case 3.} We assume $m_5=0$ and $m_1\ge 0$.

  Then
  $m_4=\max\{m_1,\ldots,m_5\}$ and we discuss the term of
  degree $2m_4$. Let $k$ be such that $m_i=m_4$ iff $4\ge i\ge k$.
  Then the degree $2m_4$-coefficient of the Laurent series
  equals
  $-u^tMu$ with
  \[
  u=(0,\ldots,\alpha_k,\ldots,\alpha_4)
  \]
  which can not be zero by the same arguments as before.

  Thus we have checked by contradiction that $(m_1,\ldots,m_5)$
  can not be different from $(0,\ldots,0)$.
\end{proof}

\begin{corollary}\label{cor-stab}
  Under the assumptions of Proposition~\ref{3.1}, let $X$ be an algebraic subvariety of $Z$
  such that $X\cap (\R^*)^5$ is not empty.

  Then the stabilizer group $Stab(X)=\{g\in G:g\cdot X=X\}$
  is discrete.
\end{corollary}

\begin{proof}
  If $Stab(X)$ is not discrete, it contains an algebraic subgroup $H$
  isomorphic to $\C^*$, i.e., given as
  \[
  H=\{(z^{m_1},\ldots,z^{m_5}):z\in\C^*\}
  \]
  with $(m_1,\ldots,m_5)
  \in\Z^5\setminus\{(0,\ldots,0)\}$.
  
  Since $X\cap (\R^*)^5$ is non-empty, there are $\alpha_i\in\R^*$
  with $(\alpha_1,\ldots,\alpha_5)\in X$.
  Then
  \[
  (\alpha_1z^{m_1},\ldots,\alpha_5z^{m_5})\in X,\,\, \forall z\in\C^*
  \]
  contradicting the preceding proposition.
\end{proof}

\begin{remark}
  The assumption that $X$ contains a real point is crucial.
  E.g., for $M=I_4$ consider
  \[
  X=\{(1,1,z,iz;2):z\in\C^*\}
  \]
  Then $X\cap (\R^*)^5$ is empty and $Stab(X)$ is one-dimensional.
\end{remark}

\begin{theorem}\label{five}
  Let $c_1,\ldots,c_5\in\H$ be given such that there is no proper
  real affine $3$-subspace of $\H$ containing all $c_i$.

  Then every slice regular function $f:\H\to\H$ with
  $f(\H)\subset\H\setminus\{c_1,\ldots,c_5\}$ is constant.
\end{theorem}

\begin{proof}
  Without loss of generality we may assume that $c_5=0$.
  By abuse of language we identify $c_i\in\H$ with $c_i\tensor 1\in\Hc$.
  Let $\left<\ ,\ \right>$ denote the complex bilinear form on $\Hc$
  induced by the euclidean scalar product on $\H\simeq\R^4$, i.e.,
  $\left<z,w\right>=\sum_i z_iw_i$.
  
  We define a holomorphic map $\phi:\Hc=\C^4\to\C^5$ by
  \begin{equation}\label{def-psi}
  \phi:
  \begin{pmatrix} z_1\\ \vdots \\ z_4 \\
  \end{pmatrix}
  \mapsto
  \begin{pmatrix}
\QQ z-2\left<z,c_1\right>+\QQ{c_1}\\
\vdots\\
\QQ z-2\left<z,c_4\right>+\QQ{c_4}\\
\QQ z\\
  \end{pmatrix}_{.}
  \end{equation}
  Observe that $\phi(z)=\overline{\phi(\bar z)}$.
  
  By assumption the vectors
  $c_1,\ldots,c_4$ form a real vector space basis for $\H$.
  It follows that there exists an invertible real $4\times 4$-matrix
  $B$ such that
  \begin{equation}
  \label{def-B}
  \begin{pmatrix} \left< z,c_1\right>\\
  \vdots\\
  \left< z,c_4\right>\\
  \end{pmatrix} = B^{-1}\cdot  z, \,\,  \forall z\in\R^4\simeq\H.
  \end{equation}
  Let $M=B^tB$. Then $M$ is a positive definite symmetric real matrix
  $M$ such that
  for every $z\in\C^4$ we have
  \[
  \QQ z =v^t\cdot M\cdot v
  \]
  if
  \[
  v= \begin{pmatrix} \left< z,c_1\right>\\
  \vdots\\
  \left< z,c_4\right>\\
  \end{pmatrix}.
  \]

  We observe that
  \[
  \phi_i(z)=\QQ{z}-2\left<z,c_i\right>+\QQ{c_i}
  \]
  for $z=(z_1,\ldots,z_4)$ and $i\in\{1,2,3,4\}$
  implies that
  \[
  \left<z,c_i\right> = -\frac 12 \left ( \phi_i(z)-\QQ z -\QQ{c_i}
  \right).
  \]
  Combined with $\phi_5(z)=\QQ{z}$ we obtain that
  \[
  \phi_5(z) = v^tMv
  \]
  for
  \[
  v_i=- \frac 1 2 \left( \phi_i(z)-\QQ{z} -\QQ{c_i}\right).
  \]

  On $\C^5$ we define an algebraic subvariety $Z$ as the zero set 
  of the function
  \begin{align*}
&  \psi(w_1,\ldots,w_4;p)
  = p - u^tM u, \quad \text{with}\\
  &  u= - \frac 12\left(w_1-p-\QQ{c_1},\ldots, w_4-p-\QQ{c_4} \right)^t.
  \\
  \end{align*}

  Due to the definition of $\psi$ it is clear that
  $\psi(w;p)=0$ if $(w,p)=\phi(z)$ for some $z\in\C^4$.

  Therefore $\phi(\C^4)\subset Z$.
  
  We claim that $\phi:\C^4\to Z$ is biholomorphic. Indeed,
  consider
  \[
  \mu: \begin{pmatrix} v_1\\ \vdots\\  v_4\\ v_5 \\
  \end{pmatrix}
  \mapsto
  B\cdot
  \begin{pmatrix}
    -\frac 12\left(v_1-\left<c_1,c_1\right>-v_5)\right)\\
    \vdots\\
    -\frac 12\left(v_4-\left<c_4,c_4\right>-v_5)\right)\\
  \end{pmatrix}
  \]
  with $B$ defined as in \eqref{def-B}.
  Due to the definitions of $\psi$ and $B$ (\eqref{def-psi}
  resp.~\eqref{def-B}) this map $\mu:Z\to\C^4$ is an inverse for
  $\psi:\C^4\to Z$. Thus $\C^4$ and $B$ are biholomorphic and
  even isomorphic as algebraic varieties.
  
  Now let $f$ be a non-constant slice regular function avoiding the values
  $c_1$,\ldots, $c_4$, $c_5=0$ and let $F:\C\to\Hc\simeq\C^4$ be its
  stem function.
  Since $\phi(\C^4)\subset Z=\{\psi=0\}$, we obtain a holomorphic
  map $g=\phi\circ F:\C\to Z$. By construction $g(z)=\overline{g(\bar z)}$
  for all $z\in\C$. Furthermore $g$ is non-constant, because $F$ is
  non-constant and $\phi$ is injective.

  Because $f:\H\to\H$ is assumed to avoid $c_i$ for every $i$, we
  know (thanks to Proposition~\ref{2.2}) that $\phi_i(F(z))\ne 0$
  for all $z\in\C$ and all $i$, i.e.,
  $\phi(F(\C))\subset Z\cap (\C^*)^5$.
  
  Thus we may apply
  Proposition~\ref{NBO} and conclude that
  there exist $\alpha_1\ldots,\alpha_5\in\R^*$
  and $(m_1,\ldots,m_5)\in\Z^5\setminus\{(0,\ldots,0)\}$
  such that $\zeta(\C^*)\subset Z$
  for 
  \[
  \zeta(z)\stackrel{def}{=}
  \left( \alpha_1z^{m_1},\ldots,\alpha_5z^{m_5}\right).
  \]
  But such a holomorphic map can not exist due to
  Proposition~\ref{3.1}
  Contradiction! Thus there is no non-constant slice regular
  function $f:\H\to\H$ avoiding all the $c_i$.
  \end{proof}
\begin{remark}
  If $f \colon \mathbb{H} \to \mathbb{H}$ is non-constant and
  slice preserving (i.e. it preserves each slice), then it can avoid only real points and at most one. \\
  If $f$ is non-constant and one-slice preserving (i.e. it preserves a unique slice), 
then it can avoid only one point on the slice which is preserved.
\end{remark}

\section{Big Picard}

In complex analysis, the ``Big Picard theorem'' states the following:
If $f$ is a holomorphic function on
$\Delta^*=\{z\in\C: 0<|z|<1\}$ with an essential singularity at $0$,
then $f$ assumes every value in $\P_1$ infinitely often with at most two
exceptions.

\begin{proposition}\label{pic-ext}
  Let $Z$ be defined as in Proposition~\ref{3.1}.
  Let $\eta$ be a holomorphic map from $\Delta^*$
  to $Z\subset (\C^*)^5\subset(\P_1)^5$
  with $\eta(\bar z)=\overline{\eta(z)}$ for all $z$.

  Then $\eta$ extends through $0$ to a holomorphic map
  to $(\P_1)^5$, i.e., the
  isolated singularity of $\eta$ at $0$ is not essential.
\end{proposition}

\begin{proof}
  Let $X$ denote Zariski closure of $\eta(\Delta^*)$
  in $Z$. Note that $\eta(z)\in(\R^*)^5$ for $z\in\R\cap\Delta^*$.
  Thus $X$ has non-trivial intersection with $(\R^*)^5$.
  It follows that $Stab(X)$ is discrete (see Corollary~\ref{cor-stab} of Section 3).
  This implies that $\eta$ extends to a holomorphic map defined on
  $\Delta$ (Proposition~\ref{NogD}).
  \end{proof}

\begin{theorem}[Quaternionic Big Picard]
  Let $\mathbb{B}$ denote the open unit ball in $\H$ and let
  $f:\mathbb{B}\setminus\{0\}\to \H$ be a slice regular function
  with stem function $F:\Delta^*\to\Hc$.
  Assume that $F$ has an essential singularity at $0$
  (i.e.~at least one of the components of $F$ has an essential
  singularity).

  Let $S$ denote the set of all $v\in\H$ for which
  the level set $f^{-1}(v)=\{q\in\H:f(q)=v\}$ is finite.

  Then $S$ is contained in an affine real hyperplane in $\H$.
\end{theorem}

\begin{proof}
  Assume the contrary. Then there are five values $c_0,\ldots,c_4$
  for which the level set is finite such that these five values
  generate $\H$ as an affine real space.
  Since  $\cup^4_{m=0}f^{-1}(c_m)$ is finite,
  we may define
  \[
  r=\min\{|q|:q\in\cup^4_{m=0}f^{-1}(c_m), q\ne 0 \},
  \quad\quad \B_r=\{q\in\H:|q|<r\}
  \]
  Now $f|_{\B_r\setminus\{0\} }$ avoids $c_0,\ldots, c_4$.
  Hence $\phi(F(z))\in(\C^*)^5\cap Z$
  for all $z\in\C$, $|z|<r$ (with $\phi$ and $Z$ defined as in
  Theorem~\ref{five}).
  Due to Proposition~\ref{pic-ext} the holomorphic map
  $\phi\circ F:\{z\in\C: 0<|z|<r\}\to Z$ extends
  to a holomorphic map with values in $(\P_1)^5$.
  But $\phi:\Hc \to Z$ is a biholomorphic map,
  whose inverse map $\phi^{-1}=\mu$ is polynomial (see the proof of Theorem 3.5).
 It follows immediately that
  $\phi^{-1}\circ (\phi\circ F)=F$ extends to a holomorphic map from
  $\Delta$ to $(\P_1)^4$.
  This yields a contradiction to our assumptions.
  \end{proof}

Since over the complex field, Picard's theorems are the global version of the local Landau's Theorem, we point out that a quaternionic Landau's Theorem for slice regular functions already exists in the literature, see \cite{bs17}.

\begin{proposition}
  For every non-constant slice regular function $f:\H\to\H$
  the image is dense
  in $\H$.
\end{proposition}
\begin{proof}
  If the image is not dense, its complement contains a
  non-empty open set. But it
  is trivially possible to choose five points in general position inside
  any given non-empty open set, leading to a contradiction with
  Theorem~\ref{five}.
\end{proof}

In particular, a bounded slice regular function $f:\H\to\H$ must be
constant, a fact which was first proved in \cite{GS}, Theorem~3.7.

\section{The example of a function avoiding $\C_I$}

Here we provide an example of a slice regular function avoiding
infinitely many values.

\begin{proposition}\label{negex}
  Let $f:\H\to\H$ be the slice regular function induced
  by the stem function
\[
F(z)=J\tensor \sin(z)+K\tensor \cos(z).
\]

Then
\[
f(\H)=\H\setminus\C_I=\{c_1+c_2I+c_3J+c_4K; c_i\in\R, (c_3,c_4)\ne (0,0)\}.
\]
\end{proposition}

\begin{proof}
  We start with some preparations concerning complex trigonometric
  functions.

  We recall that
  $\sin(iy)=i\sinh(y)$ and $\cos(iy)=\cosh(y)$ for all $y\in\R$.

For $z=x+iy$ ($x,y\in\R$) we obtain
\begin{align*}
\sin(z)=\sin(x+iy)&=\sin(x)\cos(iy)+\cos(x)\sin(iy)\\
&=\sin(x)\cosh(y)+i\cos(x)\sinh(y)
\end{align*}
and
\[
\cos(z)=\cos(x+iy)=\cos(x)\cosh(y) -i \sin(x)\sinh(y).
\]

Given $c=c_1+c_2I+c_3J+c_4K\in\H$, there exists a quaternionic number $q$ with
$f(q)=c$ iff there exists a complex number $z=x+iy$ with
\[
\QQ{F(z)-c\tensor 1}
=0.
\]

Now
\begin{align*}
&\QQ{F(z)-c\tensor 1} \\
=& \QQ{F(z)}-2\left<c\tensor 1,F(z)\right>+||c||^2\\
=& 1 -2\left( c_3\sin(z)+c_4\cos(z) \right) +||c||^2
\end{align*}
implying
\begin{equation}\label{eq-im}
\Im\left(\QQ{F(z)-c\tensor 1}\right)
=
-2\sinh(y)\left( c_3\cos(x) -c_4\sin(x) \right)
\end{equation}
and
\begin{multline}\label{eq-re}
\Re\left(\QQ{F(z)-c\tensor 1}\right)\\
= 
1 - 2 \cosh(y) \left( c_3\sin(x)+c_4\cos(x) \right) + ||c||^2.
\end{multline}

It follows that
\[
\Re\left(\QQ{F(z)-c\tensor 1}\right)=1+||c||^2\ge 1 >0
\]
if $c_3=c_4=0$. This proves that $f$ does not assume any
value in $\C_I$.

It remains to prove that all other values are assumed.

We claim: For every $c\in\H\simeq\R^4$ with $(c_3,c_4)\ne (0,0)$ there exist
$x,y\in\R$ such that $\QQ{F(x+yi)-c\tensor 1}=0$.

First we choose $x\in\R$ such that
\[
c_3\cos(x)-c_4\sin(x)=0
\]
Due to \eqref{eq-im} this guarantees that
\[
\Im(\QQ{F(x+yi)-c\tensor 1})=0.
\]

If $c_3\sin(x)+c_4\cos(x)<0$, we replace $x$ by $x+\pi$.
This ensures that
\[
c_3\sin(x)+c_4\cos(x)>0
\]

Define
\[
t= \frac{1+||c||^2}{2\left( c_3\sin x+c_4\cos x\right)}.
\]
We have to show that there exists a number $y\in\R$ with
$\cosh(y)=t$, because then it follows from \eqref{eq-im} and
\eqref{eq-re} that $\left<F(x+iy),F(x+iy)\right>=0$.

An application of the Cauchy Schwarz Inequality to the vectors
  $(c_3,c_4)$ and $(\sin(x),\cos(x))$ yields the
  inequality
  \[
  \left| c_3\sin(x)+c_4\cos(x)\right|\le \sqrt{c_3^2+c_4^2}
  \]

Using $c_3\sin(x)+c_4\cos(x)>0$ it follows that
\[
t= \frac{1+||c||^2}{2\left( c_3\sin x+c_4\cos x\right)}
\ge \frac{1+\left( c_3\sin x+c_4\cos x\right)^2}%
    {2\left( c_3\sin x+c_4\cos x\right)}\ge 1.
\]
Now $t\ge 1$ implies that there exists a real number $y$
with $\cosh(y)=t$.
This completes the proof.
\end{proof}

\section{Avoiding three points}

\begin{proposition}
  Let $c_1,c_2,c_3$ be three arbitrary quaternionic numbers.

  Then there exists a non-constant slice regular function
  $f(q)=\sum q^ka_k$ such that $f(\H)\subset\H\setminus
  \{c_1,c_2,c_3\}$.
\end{proposition}
\begin{proof}
We have seen that there exists a slice regular function
$f(q)=\sum_k q^ka_k$ with $f(\H)\subset\H\setminus\C_I$
(Proposition~\ref{negex}).

We modify this function in the following way: Let $\lambda\in\H^*$,
$p\in\H$ and let $\phi$ be a ring automorphism of $\H$.

Then we define a slice regular function $g$ by
\[
g(q)\stackrel{def}=
\left( \sum_k q^k\phi(a_k)\right) \lambda + p.
\]

For any $c\in\H$ we have
\begin{align*}
  c &= g(\phi(q))\\
  \iff c &= \phi(f(q))\lambda+p\\
  \iff \phi^{-1}(c) &= f(q)\phi^{-1}(\lambda)+\phi^{-1}(p)\\
  \iff f(q)&=\left( \phi^{-1}(c)-\phi^{-1}(p)\right) \phi^{-1}(1/\lambda).\\
\end{align*}

Let $c_1,c_2,c_3\in\H$ be three given distinct quaternionic numbers.
(Evidently it suffices to consider only the case of three
{\em distinct} numbers.)

We choose $p,\lambda,\phi$ such that:
\begin{enumerate}
\item
  $p=c_1$,
\item
  $\lambda=c_2-c_1$,
\item
  $\phi^{-1}\left( (c_3-c_1)(c_2-c_1)^{-1}\right)\in\C_I$.
\end{enumerate}

In order to verify that this is possible, let $H\in\H$ be an
imaginary unit (i.e.~$H^2=-1$)
such that
\[
 (c_3-c_1)(c_2-c_1)^{-1} \in\C_H=\R\oplus H\R.
\]
Let $\phi$ be an orientation preserving linear orthogonal
transformation of $\H$ fixing $\R$ pointwise and such that
$\phi(I)=H$.
Then $\phi$ is a ring automorphism of $\H$
satisfying $\text{(iii)}$.

It is easily verified that
\[
\left( \phi^{-1}(c_i)-\phi^{-1}(p)\right) \phi^{-1}(1/\lambda)\in\C_I
\]
for all three indices $i\in\{1,2,3\}$.
Since $f$ avoids values in $\C_I$, it follows that $g$ avoids the
three values $c_1,c_2,c_3$.
\end{proof}

\begin{remark}
  Since any $2$-dimensional real affine subspace $P$ of $H\simeq\R^4$ is
  spanned by three points, it follows form the above that there exists
  an entire slice regular function $f:\H\to\H$
  such that $f(\H)=\H\setminus P$.
\end{remark}

{\bf Open Problem: } \\
Is or isn't there a non-constant slice regular entire function of $\mathbb{H}$ avoiding four general points? 

\section{Octonions}

In view of the results of \cite{ghiloniperotti}, in particular theorem~17, one may easily
modify our arguments in order to obtain a Picard theorem for the
algebra of octonions, namely:
\begin{theorem}
  
  For every non-constant {\em slice regular function}
  $f:{\mathbb O}\to{\mathbb O}$
  the set ${\mathbb O}\setminus f({\mathbb O})$ is contained
  in a real affine hyperplane of $\mathbb O$.
\end{theorem}

\end{document}